\documentclass[a4paper,11pt]{article}
\usepackage[T1]{fontenc}
\usepackage[utf8]{inputenc}
\usepackage[english]{babel}
\usepackage{verbatim}
\usepackage{savesym}
\usepackage{amsmath}
\usepackage{amssymb}
\usepackage{amsfonts}
\usepackage{booktabs}
\savesymbol{iint}
\usepackage{amsthm}
\savesymbol{openbox}
\usepackage{txfonts}
\usepackage{dsfont}
\restoresymbol{TXF}{iint}
\restoresymbol{TXF}{openbox}
\usepackage[all]{xy}
\usepackage{graphicx}
\usepackage{emptypage}
\usepackage{lscape}

\newenvironment{sistema}%
{\left\lbrace\begin{array}{@{}l@{}}}%
{\end{array}\right.}

\newcommand{\ZZ}{\ensuremath{\mathds{Z}}}
\newcommand{\QQ}{\ensuremath{\mathds{Q}}}

\newcommand{\CC}{\ensuremath{\mathds{C}}}

\newcommand{\PP}{\ensuremath{\mathds{P}}}
\newcommand{\HH}{\ensuremath{\mathcal{H}}}

\newcommand{\cfs}{\ensuremath{\mathcal{S}}}

\newcommand{\SL}{\ensuremath{\mathrm{SL}}}

\newcommand{\Gammait}{\ensuremath{\mathit{\Gamma}}}

\newcommand{\Span}{\ensuremath{\mathrm{span}}}

\newcommand{\num}{\mathop{\mathrm{Numerator}}}
\newcommand{\den}{\mathop{\mathrm{Denominator}}}

\theoremstyle{definition}
\newtheorem{definizione}{Definition}[section]
\theoremstyle{plain}
\newtheorem{lemma}[definizione]{Lemma}
\theoremstyle{plain}

\theoremstyle{plain}

\theoremstyle{plain}
\newtheorem{proposizione}[definizione]{Proposition}
\theoremstyle{plain}
\newtheorem{result}[definizione]{Result}
\theoremstyle{plain}

\theoremstyle{remark}
\newtheorem{remark}[definizione]{Remark}
\theoremstyle{remark}
\newtheorem{algoritmo}[definizione]{Algorithm}

\title{Equations and Rational Points \\ of the Modular Curves $X_0^+(p)$}
\author{Pietro Mercuri}
\date{}

\begin{document}

\maketitle

\begin{abstract}
Let $p$ be an odd prime number and let $X_0^+(p)$ be the quotient of the classical modular curve $X_0(p)$ by the action of the Atkin-Lehner operator $w_p$. In this paper we show how to compute explicit equations for the canonical model of $X_0^+(p)$. Then we show how to compute the modular parametrization, when it exists, from $X_0^+(p)$ to an isogeny factor $E$ of dimension 1 of its jacobian $J_0^+(p)$. Finally we show how use this map to determine the rational points on $X_0^+(p)$ up to a large fixed height.
\end{abstract}

\section{Introduction}

The classical problem to determine the rational points of the modular curve $X_0(p)$, when $p$ is a prime number, has been solved by Mazur in his paper \cite{Maz76} in 1975. In his work, Mazur proved, for $p>37$, that the only rational points of the curve $X_0(p)$ are cusps and CM points. In the same paper Mazur also refers to the interesting case of the modular curve $X_0^+(p)$. This curve is the quotient of $X_0(p)$ by the Atkin-Lehner involution $w_p$. See Galbraith \cite[Section 2]{Gal99}. For these curves a similar result is expected: for sufficiently large $p$ the only rational points of the curve $X_0^+(p)$ are cusps and CM points.

The genus $g$ of $X_0^+(p)$ growth with $p$, so for a fixed $g$ we have only finitely many curves. The cases $g$ equal to $0$ or $1$ are well known. If $g$ is equal to $0$, then $X_0^+(p)$ is isomorphic to $\PP^1$ and there are infinitely many rational points. If $g$ is equal to~$1$, then $X_0^+(p)$ is an elliptic curve whose Mordell-Weil group is torsion-free and has rank~$1$. Let us focus on the case $g>1$. Galbraith in his papers \cite{Gal96} and \cite{Gal99} describes a method to find explicit equations for the curves $X_0^+(p)$. If $X_0^+(p)$ is hyperelliptic he uses an \emph{ad hoc} technique. We recall that by \cite{HH96} that the curve $X_0^+(p)$ is hyperelliptic if and only if $g=2$. If $X_0^+(p)$ is not hyperelliptic, Galbraith uses the canonical embedding $\varphi\colon X_0^+(p) \hookrightarrow \PP^{g-1}$ to find equations for the curves. He finds explicit equations for all curves $X_0^+(p)$ with $g\le 5$. Here we present the following result.
\begin{result}
We use the method of Galbraith to find explicit equations for the canonical model of all the modular curves $X_0^+(p)$ with genus $g=6$ or $7$.
\end{result}
Equations for the thirteen curves of genus $6$ or $7$ are listed in Section \ref{sec:listeq}. The equations that we find have very small coefficients and have good reduction for each prime number $\ell\neq p$.

The technique used here play an important role in \cite{Mer15} in finding explicit equations defining the modular curves $X_{ns}^+(p)$. These curves are the modular curves associated to the normalizer of a nonsplit Cartan subgroup of the general linear group over a finite field with $p$ elements. Few is still known about them and they are harder to study with respect to modular curves $X_0^+(p)$ considered here.

Galbraith uses the equations of the curves $X_0^+(p)$ to look for rational points $P$ up to some bounds on the naive projective height $H$. See Silverman \cite[Chapter VIII, Section 5]{Sil86} for more details about the naive height on projective space. Using a projection onto a plane model, possibly singular, we check that if $X_0^+(p)$ has genus $g=6$ or $7$, then it has no rational points $P$, except the cusps and the CM points, for which $H(P)\le 10^4$. When the jacobian variety $J_0^+(p)$ of $X_0^+(p)$ has an isogeny factor over $\QQ$ of dimension 1, we improve this bound significantly. In particular we prove the following result.
\begin{result}
For the primes $p=163,197,229,269$ and $359$, the only rational points $P$ on $X_0^+(p)$ of naive height $H(P)\le 10^{10000}$ are cusps or CM points.
\end{result}

\section{The method to get explicit equations} \label{sec:algoeq}

In this section we show how to get explicit equations of the canonical model for modular curves. The final output is a list of quadrics with very small integer coefficients.

Let $\Gammait$ be a congruence subgroup of $\SL_2(\ZZ)$, let $\cfs_2(\Gammait)$ be the $\CC$-vector space of the cusp forms of weight $2$ with respect to $\Gammait$, and let $f_1,\ldots,f_g$ be a $\CC$-basis of $\cfs_2(\Gammait)$ where $g$ is the genus of the modular curve $X(\Gammait)$. If \mbox{$g>2$}, we use the canonical embedding to get the canonical model for $X(\Gammait)$. We know that $X(\Gammait)$ is isomorphic as compact Riemann surfaces to $\Gammait \backslash \HH^*$, where $\HH^*$ is the extended complex upper half plane. We also know that $\cfs_2(\Gammait)$ is isomorphic to the $\CC$-vector space of holomorphic differentials $\Omega^1(X(\Gammait))$ with respect to the map $f(\tau)\mapsto f(\tau)d\tau$. See Diamond and Shurman \cite{DS05} for more details about these results. Using these isomorphisms, the canonical embedding can be expressed in the following way
\begin{align*}
\varphi\colon X(\Gammait) & \to \PP^{g-1}(\CC) \\
\Gammait\tau & \mapsto (f_1(\tau):\ldots:f_g(\tau)),
\end{align*}
where $\tau\in\HH^*$. The Enriques-Petri Theorem (see Griffith, Harris \cite[Chapter 4, Section 3, pag. 535]{GH78} or Saint-Donat \cite{Sd73}), states that the canonical model of a complete nonsingular non-hyperelliptic curve is entirely cut out by quadrics and cubics. Moreover, it is cut out by quadrics if it is neither trigonal, nor a quintic plane curve with genus exactly~$6$.

When $X(\Gammait)$ can be defined over $\QQ$, we can look for equations defined over~$\QQ$. The Enriques-Petri Theorem is proved over algebraically closed fields, but its proof can be suitable modified to work over $\QQ$. Alternatively, if we find equations defined over $\QQ$, then we can check by MAGMA that their zero locus $Z$ is an algebraic curve with the same genus as $X(\Gammait)$. An application of the Hurwitz genus formula for genus $g>1$, to the morphism $\varphi\colon X(\Gammait)\to Z$ implies that $\varphi$ is an isomorphism.

Finding equations for $Z$ is equivalent to finding generators of the homogeneous ideal $\mathcal{I}$ defining it in $\PP^{g-1}$. Let $\mathcal{I}_d$ be the set of homogeneous elements of $\mathcal{I}$ of degree exactly $d$. We explain how to find generators of $\mathcal{I}$ that belong to $\mathcal{I}_d$ for a fixed $d$. By the Enriques-Petri Theorem and by Hasegawa and Hashimoto \cite{HH96}, we know that if $g>2$, the ideal $\mathcal{I}$ is generated by elements in $\mathcal{I}_d$ for $d=2,3$. Now, we fix the degree $d$ and we suppose that we know the first $m$ Fourier coefficients of the $q$-expansions of a basis $\mathcal{B}=\{f_1,\dots,f_g\}$ of $\Omega^1(X(\Gammait))$, where $m>d(2g-2)$. This condition on $m$ guarantees that if we have a polynomial $F$ with rational coefficients and $g$ unknowns such that $F(f_1,\ldots,f_g)\equiv 0 \pmod{q^{m+1}}$, then $F(f_1,\ldots,f_g)=0$. See \cite[Section 2.1, Lemma 2.2, pag 1329]{BGGP05} for more details. We also assume that the Fourier coefficients of the basis $\mathcal{B}$ are algebraic integers.

There is a number field $K$, of degree~$D$ over~$\QQ$, which contains all the coefficients of all the elements of $\mathcal{B}$. Moreover, if the Fourier coefficients are algebraic integers, they have integer coordinates with respect to a suitable chosen basis of $K$ over $\QQ$. If $a_n(f_i)$ is the $n$-th Fourier coefficient of $f_i$, we denote the $k$-th coordinate of $a_n(f_i)$ in $K$ by $c(i,n,k)\in\ZZ$, where $k=1,\ldots,D$. We can associate to $\mathcal{B}$ a matrix $M$ whose integer entries are the coordinates of the Fourier coefficients where each row $i$ corresponds to an element $f_i$ of $\mathcal{B}$. An explicit way to set the entries of $M=(u_{ij})$ is to choose $u_{ij}=c(i,n,k)$, where $j=(n-1)D+k,$ for $i=1,\ldots,g$, for $n=1,\dots,m$ and for $k=1,\ldots,D$. The matrix $M$ has $g$ rows and $mD$ columns, where $m$ is the number of the first Fourier coefficients used and $g$ is the cardinality of the basis. Since $m>d(2g-2)$ and $g>0$, we always have $m\geq g$, hence the rank of $M$ is $g$. To find equations defining $X(\Gammait)$ we compute all the monomials $F_j$ of degree $d$ where the indeterminates are the elements of $\mathcal{B}$. The elements of a $\ZZ$-basis of the space $S$ of the solutions of the homogeneous linear system in the unknowns $F_j$ are the coefficients of the desired equations.

We are interested in reducing the size of the coefficients of these equations and minimizing the number of primes $\ell$ such that the model has bad reduction modulo~$\ell$. To reduce the size of the coefficients, we apply the LLL-algorithm first to $M$ and then to the $\ZZ$-basis of $S$. We know that if the rank of $M$ modulo $\ell$ is less than $g$, then the canonical model of the curve that we find, is singular modulo $\ell$. We say that $M$ is \emph{optimal}, if the rank of $M$ modulo $\ell$ is exactly $g$ for each prime~$\ell$. In Algorithm~\ref{alg:badprimes} below, within a more general setting, we describe how to modify $\mathcal{B}$ to make $M$ optimal.

Now, we describe how to find the primes $\ell$ such that the canonical model has bad reduction. Let $\mathcal{I}_{\text{jac}}$ be the ideal generated by the polynomials defining the curve and by all the determinants of order $g-2$ of the jacobian matrix of the curve. We compute the elimination ideals $\mathcal{J}_i:=\mathcal{I}_{\text{jac}}\cap \QQ[x_i]$, for $i=1,\ldots,g$. If $\mathcal{J}_i\neq 0$, it turns out to be generated by $\lambda x_i^n$ with $\lambda,n\in\ZZ_{>0}$. The curve has bad reduction modulo any prime~$\ell$ such that $\ell\mid \lambda$.

\subsection{Algorithm}

Let $g$ and $m$ be positive integers such that $g\leq m$ and let $v_1,\ldots,v_g\in\ZZ^m$ be linearly independent vectors over $\QQ$. We describe a method to find a basis over $\ZZ$ of the lattice $L:=\Span_{\QQ}\{v_1,\ldots,v_g\} \cap \ZZ^m$. Let $L':=\Span_{\ZZ}\{v_1,\ldots,v_g\}$, so $L'$ is a subgroup of $L$, and let $J:=[L:L']$. We want to modify $L'$ and its basis until we have $J=1$ and so $L'=L$.

\begin{lemma} \label{lem:pdep}
Let the notation as above and let $p$ be a prime number. We have that $v_1,\ldots,v_g$ are linearly dependent in $\ZZ^m/p\ZZ^m$ if and only if $p\mid [L:L']$.
\end{lemma}
\begin{proof}
We have that $p\mid [L:L']=\#(L/L')$ if and only if there is an element of $L/L'$ of order exactly $p$. This is equivalent to have an element $v\in L\setminus L'$ such that $pv\in L'$. The condition $v\in L$ implies that there are $\alpha_1,\ldots,\alpha_g\in\QQ$ such that $v=\alpha_1 v_1+\ldots+\alpha_g v_g$. The condition $pv\in L'$ implies that $p\alpha_1,\ldots,p\alpha_g$ are integers. The independence over $\QQ$ implies that $p\alpha_1,\ldots,p\alpha_g$ are not all zero. Finally, the condition $v\notin L'$ implies that $p\alpha_1,\ldots,p\alpha_g$ are not all zero modulo $p$.
\end{proof}
Let $M\in\ZZ^{m\times g}$ be the matrix whose columns are the vectors $v_1,\ldots,v_g$. It follows from Lemma \ref{lem:pdep} above that if a prime $p$ divides the index $J$, then $p$ divides the determinant $\Delta$ of any $g\times g$ submatrix of $M$.

\begin{algoritmo} \label{alg:badprimes}
\begin{description}

\item[Step 0] Choose a $g\times g$ submatrix of $M$ with nonzero determinant~$\Delta$. Set $\mathcal{P}:=\{\text{prime numbers dividing $\Delta$}\}$. Go to Step~1.

\item[Step 1] If $\mathcal{P}=\varnothing$ the algorithm terminates. If $\mathcal{P}\neq \varnothing$ go to Step~2.

\item[Step 2] Let $p\in \mathcal{P}$. If $v_1,\ldots,v_g$ are linearly dependent in $\ZZ^m/p\ZZ^m$, we have that $p\mid J$, and go to Step~3. Else set $\mathcal{P}:=\mathcal{P}\setminus \{p\}$ and go to Step~1.

\item[Step 3]  If $v_1,\ldots,v_g$ are linearly dependent in $\ZZ^m/p\ZZ^m$, up to reordering the $v_i$'s, there are, $\alpha_2,\ldots,\alpha_g\in\ZZ$ such that $v_1+ \alpha_2 v_2+\ldots+\alpha_g v_g=pv$. We replace $v_1$ by $v=\frac{1}{p}(v_1+ \alpha_2 v_2+\ldots+\alpha_g v_g)\in L\setminus L'$ and go to Step~2.

\end{description}
\end{algoritmo}

\begin{remark}
To reduce the number of primes to check, one can apply Gauss elimination to make $\Delta$ minimal.
\end{remark}

\begin{remark}
The algorithm terminates in finitely many steps. In fact the substitution $v_1\mapsto \frac{1}{p}(v_1+ \alpha_2 v_2+\ldots+\alpha_g v_g)$ implies that
\[
\Delta=\det(v_1,\ldots,v_g)\mapsto \det\left(\frac{1}{p}\left(v_1+\sum_{i=2}^g \alpha_i v_i\right),v_2,\ldots,v_g\right)=\frac{1}{p}\det(v_1,\ldots,v_g)=\frac{\Delta}{p}.
\]
Hence, after finitely many iterations of the algorithm the rank of $M$ modulo $p$ is $g$ for each prime $p$.
\end{remark}

\section{Computing the expected rational points}

We have a simple moduli interpretation of certain rational points of the modular curves $X_0^+(p)$. The \emph{expected} rational points on $X_0^+(p)$ are the points corresponding to the unique cusp or to elliptic curves with complex multiplication such that $p$ is split or ramified inside the endomorphism ring of the elliptic curve itself. We call a rational point \emph{exceptional} if it is not one of the expected ones.

Here we describe a way to find numerically the expected rational points on a modular curve $X_0^+(p)$, for $p$ an odd prime. 
We assume to know the first $m$ Fourier coefficients of a basis $f_1,\ldots,f_g$ of $\Omega^1(X_0^+(p))$, where $g$ is the genus of the modular curve. It is well known, see Stark \cite{Sta66}, that the orders in imaginary quadratic number fields with class number $1$ are the ones with discriminant
\[
\Delta=-3,-4,-7,-8,-11,-12,-16,-19,-27,-28,-43,-67,-163.
\]
The corresponding orders can be written in the form $\ZZ+\ZZ\tau$, where
\[
\tau=\begin{cases}
\frac{1+i\sqrt{|\Delta|}}{2} & \text{if }\Delta \text{ is odd}, \\
\frac{i\sqrt{|\Delta|}}{2} & \text{if }\Delta \text{ is even}.
\end{cases}
\]

Let $E$ be an elliptic curve with complex multiplication such that the discriminant $\Delta_E$ of its endomorphism ring $\mathcal{O}_E$ is a class number 1 discriminant. This means that $\Delta_E$ is in the list above. If $\Delta_E$ is a square modulo $p$ we can associate to $E$ a point on $X_0^+(p)$, see Galbraith \cite{Gal99}. The unique cusp is always rational. Let $\{1,\tau_E\}$ be the basis of the order $\mathcal{O}_E$ such that $\tau_E$ is defined as above. There is a suitable element $\hat{\tau}$ in the $\SL_2(\ZZ)$-orbit of $\tau_E$ in $\HH$ such that $P=(f_1(\hat{\tau}):\ldots:f_g(\hat{\tau}))$ is a rational point for the modular curve. If we know $\hat{\tau}$, we can evaluate $f_1(\hat{\tau}),\ldots,f_g(\hat{\tau})$ numerically using the $q$-expansions. If $m$ is large enough, it is quite easy recognize the rational coordinates of $P$. To compute the coordinates of the cusp it is enough take $\hat{\tau}=i\infty$, this means to set $q=0$ in the $q$-expansions. Now we explain how to compute $\hat{\tau}$ for the CM points.

We recall the moduli interpretation of the points of $X_0^+(p)$. Let $E'$ be an elliptic curve and let $C$ be a cyclic subgroup of order $p$ of the $p$-torsion subgroup $E'[p]$. We know that a point on $X_0^+(p)$ is an unordered pair $\{(E',C),(E'/C,E'[p]/C)\}$, where $E'/C$ is the unique elliptic curve, up to isomorphism, that is the image of the unique isogeny of $E'$ with kernel~$C$. Over $\CC$ this is $\{(\CC/\Lambda,\frac{1}{p}\ZZ+\Lambda),(\CC/\Lambda',\frac{1}{p}\ZZ+\Lambda')\}$ for some $\tau\in\HH$ such that $\Lambda=\ZZ+\tau\ZZ$ and $\Lambda'=\ZZ+\frac{-1}{p\tau}\ZZ$. If $\Delta_E$ is a square modulo~$p$, we know there is a principal prime ideal $\mathfrak{p}$ of $\mathcal{O}_E$ such that $(p)=\mathfrak{p}\bar{\mathfrak{p}}$. Let $\alpha$ and $\bar{\alpha}$ be generators of $\mathfrak{p}$ and $\bar{\mathfrak{p}}$ respectively. The cyclic subgroups $C$ and $E'[p]/C$ are the kernels of the multiplication by $\alpha$ and $\bar{\alpha}$. Let $(E',C)=(\CC/\Lambda_E,\frac{1}{\bar{\alpha}}\ZZ+\Lambda_E)$, where $\Lambda_E=\ZZ+\tau_E\ZZ$. The point $\{(\CC/\Lambda_E,\frac{1}{\bar{\alpha}}\ZZ+\Lambda_E),(\CC/\Lambda_E,\frac{1}{\alpha}\ZZ+\Lambda_E)\}$ is rational on~$X_0^+(p)$.

Now, we want to find $\hat{\tau}$ such that  $(\CC/\Lambda_E,\frac{1}{\bar{\alpha}}\ZZ+\Lambda_E)=(\CC/\hat{\Lambda},\frac{1}{p}\ZZ+\hat{\Lambda})$, where $\hat{\Lambda}=\ZZ+\hat{\tau}\ZZ$. Let $c,d\in\ZZ$ such that $\alpha=c\tau_E+d$. Then we can find $a,b\in\ZZ$ such that $ad-bc=1$ and so we have a matrix $\hat{\gamma}=\begin{pmatrix}
a & b \\
c & d
\end{pmatrix}\in\SL_2(\ZZ)$. The transformation $\hat{\tau}=\hat{\gamma} \tau_E$ correspond to the isogeny $z+\Lambda_E\mapsto \frac{z}{\alpha}+\hat{\Lambda}$ for every $z+\Lambda_E\in\CC/\Lambda_E$. Hence $\frac{1}{\bar{\alpha}}+\Lambda_E\mapsto \frac{1}{\alpha\bar{\alpha}}+\hat{\Lambda}=\frac{1}{p}+\hat{\Lambda}$ and the group $\frac{1}{\bar{\alpha}}\ZZ+\Lambda_E$ maps to $\frac{1}{p}\ZZ+\hat{\Lambda}$ and we are done.

\section{Computing the modular parametrization} \label{sec:compEmap}

In this section we assume to know the first $m$ Fourier coefficients of the modular forms involved, where $m$ is "large enough" for our purposes. Let $p$ be a prime number and let $J_0(p)$ be the jacobian variety of the modular curve $X_0(p)$. If there is an isogeny factor $E$ over $\QQ$ of dimension $1$ of $J_0(p)$, we know there is a non-constant morphism $\phi\colon X_0(p)\to E$ defined over $\QQ$. Let $f$ be the normalized eigenform in $\cfs_2(\Gammait_0(p))$ associated to the isogeny class of $E$. Let $\mathds{T}$ be the Hecke algebra over $\ZZ$ and let $I_f=\{T\in\mathds{T}: Tf=0\}$. We know that if the Fourier coefficients of $f$ belong to $\QQ$, the associated abelian variety $J_0(p)/I_f J_0(p)$ is an elliptic curve called \emph{optimal curve} (or \emph{strong Weil curve}). See Diamond and Shurman \cite[Sections 6.5 and 6.6]{DS05} for more details about these topics. If $E=J_0(p)/I_f J_0(p)$, then we call $\phi$ the \emph{modular parametrization with respect to $X_0(p)$}. The degree of the morphism $\phi$ is called the \emph{modular degree}. Using the identification of $X_0(p)$ with $\Gammait_0(p)\backslash\HH^*$, we can write
\begin{align*}
\phi\colon \Gammait_0(p)\backslash\HH^* &\to E(\CC) \\
\Gammait_0(p)\tau &\mapsto \phi(\tau)=(\phi_x(\tau),\phi_y(\tau)),
\end{align*}
where $\tau\in\HH^*$ and $\phi_x(\tau)$ and $\phi_y(\tau)$ are Fourier series with rational coefficients in the indeterminate $q=e^{2\pi i \tau}$.

\begin{lemma} \label{lem:factorparmod}
If $E$ is an elliptic curve over $\QQ$ of conductor a prime $p$ and with negative sign of the functional equation of the associated L-function, then the modular parametrization $\phi\colon X_0(p)\to E$ factors as $\phi=\phi_+\circ\pi_p$, where $\pi_p\colon X_0(p)\to X_0^+(p)$ is the natural projection and $\phi_+\colon  X_0^+(p)\to E$ is the modular parametrization with respect to $X_0^+(p)$. Moreover $\phi_+$ is defined over $\QQ$.
\end{lemma}
\begin{proof}
Just apply Galois Theory to the associated function fields and use invariance under the Atkin-Lehner involution of the differentials of $X_0^+(p)$.
\end{proof}

Let $X_0^+(p)$ a modular curve of genus $g$ with an isogeny factor $E$ over $\QQ$ of dimension $1$ of the jacobian variety $J_0^+(p)$. We denote by $f_1,\ldots,f_g$ some $g$ linearly independent newforms of $\cfs_2(\Gammait_0(p))$ with eigenvalue $+1$ with respect to the Atkin-Lehner operator~$w_p$. This is a basis of $\cfs_2(\Gammait^+_0(p))$. The modular parametrization $\phi_+\colon  X_0^+(p)\to E$ is locally given by four homogeneous polynomials $p_x,q_x,p_y,q_y\in\ZZ[x_1,\ldots,x_g]$, such that $p_x$ and $q_x$ have the same degree, $p_y$ and $q_y$ have the same degree and
\begin{align} \label{eq:sistema}
\begin{sistema}
\phi_x(\tau)= \frac{p_x(f_1(\tau),\ldots,f_g(\tau))}{q_x(f_1(\tau),\ldots,f_g(\tau))} \\
\phi_y(\tau)= \frac{p_y(f_1(\tau),\ldots,f_g(\tau))}{q_y(f_1(\tau),\ldots,f_g(\tau))},
\end{sistema}
\end{align}
for every $\tau$ representative of $\Gammait_0(p)\tau$ chosen in a suitable open set of $X_0(p)$.

To find these polynomials explicitly we use some linear algebra. The main difficulty is that we don't know exactly the degrees of the polynomials $p_x,q_x,p_y,q_y$, but we know the degree of the morphism $\phi_+$. In Cremona \cite{Cre92} or in the LMFDB online database \cite{lmfdb} we can find the value of $\deg\phi$ and we have $\deg\phi_+=\frac{1}{2}\deg\phi$. For increasing degrees from $1$ to $\deg{\phi_+}$, we apply the following procedure to look for nontrivial relations. We rewrite the first equation as
\[ \label{eq:ctrimg}
p_x(f_1(\tau),\ldots,f_g(\tau))- q_x(f_1(\tau),\ldots,f_g(\tau))\phi_x(\tau)=0.
\]
Let $r=\begin{pmatrix}
g +d_x-1 \\
d_x
\end{pmatrix}$ be the number of monomials of degree $d_x$. Let $\mathcal{I}=\{F_1,\ldots,F_r\}$ be the set of the monomials obtained as product of $d_x$ elements of the basis $f_1,\ldots,f_g$, where repetitions are allowed, and let $\mathcal{J}=\{G_1,\ldots,G_r\}$ be the set with the same elements of $\mathcal{I}$ multiplied by $\phi_x$. Since the Fourier coefficients of $\phi_x,f_1,\ldots,f_g$ are algebraic numbers, the coefficients of the elements of $\mathcal{I}$ and $\mathcal{J}$ are also algebraic numbers and there is a number field $K_f$ of finite degree over $\QQ$, which contains all these Fourier coefficients. Multiplying the elements of $\mathcal{I}$ and $\mathcal{J}$ by a suitable algebraic number, we can assume that the Fourier coefficients belong to the ring of integers of $K_f$ and that they have integer coordinates with respect to a suitable basis of~$K_f$. Hence, there are $\alpha_1,\ldots,\alpha_r,\beta_1,\ldots,\beta_r\in\ZZ$ such that
\[
\begin{sistema}
\alpha_1 F_1+\ldots+\alpha_r F_r+\beta_1 G_1+\ldots+\beta_r G_r=0, \\
\alpha_1 F_1+\ldots+\alpha_r F_r\neq 0, \\
\beta_1 G_1+\ldots+\beta_r G_r\neq 0.
\end{sistema}
\]
These $\alpha_i$'s and $\beta_i$'s are the coefficients of $p_x$ and $q_x$ respectively. Different choices of $\alpha_1,\ldots,\alpha_r,\beta_1,\ldots,\beta_r$ correspond to the same map defined on different open sets. The same method is applied to the second equation of (\ref{eq:sistema}) to find the coefficients of $p_y$ and $q_y$.
\begin{remark}
We find the optimal curve with conductor $p$ in Cremona's tables in \cite{Cre92} or in the LMFDB online database \cite{lmfdb}. We use PARI to compute the $q$-expansions of the two components $\phi_x(\tau)$ and $\phi_y(\tau)$ of the modular parametrization of $X_0(p)$.
\end{remark}

\section{An estimation on heights} \label{sec:upbnd}

Let $n$ be a positive integer and let $X$ be a complete nonsingular curve defined over $\QQ$ in $\PP^{n-1}$ and with an elliptic curve $E$ as isogeny factor over $\QQ$ of its jacobian variety. In this section we suppose to know a Weierstrass equation of $E$, an explicit formula for the morphism \mbox{$\phi\colon X \to E$} defined over $\QQ$ and generators for the Mordell-Weil group of $E$. Writing $\phi$ in more explicit terms, we have
\[
\phi(x_1,\ldots,x_n)=\left(\frac{p_x(x_1,\ldots,x_n)}{q_x(x_1,\ldots,x_n)}, \frac{p_y(x_1,\ldots,x_n)}{q_y(x_1,\ldots,x_n)}\right),
\]
where $p_x,q_x$ are homogeneous polynomials with integer coefficients of degree $d_x$ and $p_y,q_y$ are homogeneous polynomials with integer coefficients of degree $d_y$. Let $r$ be the number of monomials of degree $d_x$ and let $\alpha_i$ and $\beta_i$, for \mbox{$i=1,\ldots, r$}, be the coefficients of $p_x$ and $q_x$ respectively. Moreover, let \mbox{$\alpha:=\log\max\left\{\sum_i\left|\alpha_i\right|,\sum_i\left|\beta_i\right|\right\}$} a constant that we use in the Proposition \ref{pro:upbnd} below.

In the paper \cite{Sil90}, Silverman gives an estimate of the difference between the canonical height and the Weil height on an elliptic curve. We use the Theorem~1.1 of \cite{Sil90} to prove the following proposition.
\begin{proposizione} \label{pro:upbnd}
Let the notation be as above and let $Q\in\PP^{n-1}$. Let $H(Q)$ be the naive height of $Q$ and let $h(Q):=\log H(Q)$. If $Q\in X(\QQ)$, then $\phi(Q)\in E(\QQ)$ and
\[
\hat{h}(\phi(Q))\le \mu(E)+1.07+\frac{1}{2}\Big(\alpha+d_x h(Q)\Big),
\]
where the quantity $\mu(E)$ is defined in Theorem 1.1 of \cite{Sil90}, the function $\hat{h}$ is the canonical height on $E$, the number $d_x$ is the degree of the homogeneous polynomials $p_x$ and $q_x$ defined above and $\alpha$ is the constant defined above.
\end{proposizione}
\begin{proof}
Let $Q$ be a rational point on $X$. Let $(x_1:\ldots:x_n)$ its coprime integers coordinates in $\PP^{n-1}$. We denote by $P$ the image of this rational point under $\phi$, hence we have $\phi(x_1,\ldots,x_n)=P=(P_x,P_y)$, where
\[
P_x=\frac{p_x(x_1,\dots,x_n)}{q_x(x_1,\dots,x_n)}.
\]
If $|x_i|\leq H(Q)$, by the definition of height we have
\[
H(P_x)\leq H(Q)^{d_x} \max\left\{\sum_i\left|\alpha_i\right|,\sum_i\left|\beta_i\right|\right\}.
\]
Taking the logarithm in the previous inequality we get $h(P_x)\leq\alpha+d_x h(Q)$. Now, using the second inequality in Theorem 1.1 of \cite{Sil90}, we are done.
\end{proof}
\begin{remark}
Since $\phi$ is defined over $\QQ$, every rational point on the curve $X$ must go in a rational point on the elliptic curve $E$. Therefore, to find rational points on $X$, it is enough to search in the preimage $\phi^{-1}(P)$ letting $P$ run over the rational points of~$E$. 
\end{remark}
\begin{remark}
In all our cases the Mordell-Weil group has rank $1$ and is torsion-free, so it is generated by one point. One can find a generator in Cremona's tables in \cite{Cre92} or in the LMFDB online database \cite{lmfdb}.
\end{remark}
\begin{remark} \label{rem:upbnd}
If the Mordell-Weil group of $E$ has rank $1$ with generator $P_0$ and is torsion-free, every rational point $P$ of $E$ has the form $P=kP_0$ for some $k\in\ZZ$. In this case we have
\[
\hat{h}(P)=\hat{h}(kP_0)=k^2\hat{h}(P_0)\leq\mu(E)+1.07+\frac{1}{2}\Big(\alpha+d_x h(Q)\Big),
\]
by properties of the canonical height. See \cite[Chapter VIII, Section 9]{Sil86} for more details about the canonical height.
\end{remark}

\section{Tables} \label{sec:listeq}

In \ref{subsec:eq} we list the equations for the canonical model of the modular curves $X_0^+(p)$ for primes $p$ such that the genus $g$ is $6$ or $7$. We also list the expected rational points. The models all have good reduction modulo every prime $\ell\neq p$. In \ref{subsec:modpar} we give an explicit morphism $\phi_+\colon X_0^+(p)\to E$, whenever the jacobian variety $J_0^+(p)$ of $X_0^+(p)$ admits a $1$-dimensional isogeny factor $E$ over $\QQ$.

\subsection{Equations}\label{subsec:eq}

Here we list the genus, the equations of the canonical model and the expected rational points with the corresponding discriminants.
\begin{itemize}
\item Curve $X_0^+(163)$. Genus $g=6$. Equations of the canonical model in $\PP^{5}$
\[
\begin{sistema}
x_1x_5+x_2x_3+x_2x_4-x_2x_5+x_2x_6=0 \\
x_1^2-x_1x_2+x_1x_5-x_2x_5+x_3^2+x_3x_4=0 \\
x_2x_4-x_2x_5+x_3x_5=0 \\
-x_1^2+x_1x_2-x_1x_4+x_2x_4 + x_3x_6=0 \\
x_1x_3+x_1x_5+x_1x_6+x_2x_3-x_2x_5+x_4x_5=0 \\
-x_1^2+x_1x_2+x_1x_3+x_1x_4-x_1x_5+x_1x_6+x_2x_5-x_3^2+x_4^2+x_5x_6=0.
\end{sistema}
\]
\[
\begin{array}{|cc|cc|}
\toprule
\text{Rational point} & \text{Disc.} & \text{Rational point} & \text{Disc.} \\
\midrule
(0:0:0:0:0:1) & \text{cusp} & (0:1:1:0:1:0) & -12 \\
(24:10:13:15:-50:42) & -163 & (0:1:0:0:0:0) & -11 \\
(1:0:1:-2:0:-1) & -67 & (0:0:1:-1:0:0) & -8 \\
(1:1:2:-2:2:0) & -28 & (1:1:0:0:0:0) & -7 \\
(1:1:0:1:1:-1) & -27 & (2:-1:3:-4:-1:-2) & -3 \\
(0:0:0:0:1:0) & -19 & & \\
\bottomrule
\end{array}
\]
\item Curve $X_0^+(193)$. Genus $g=7$. Equations of the canonical model in $\PP^{6}$
\[
\begin{sistema}
-x_1^2-x_1x_4-x_1x_6-x_1x_7+x_2x_7+x_3^2+x_3x_4=0 \\
x_1x_2+x_1x_4-x_1x_5+x_1x_7-x_2x_3-x_2x_4-x_2x_7+x_3x_5=0 \\
x_1x_2-x_1x_5-x_2x_7+x_3x_6=0 \\
-x_1x_2-x_1x_3+2x_1x_5+x_1x_6+2x_2x_3+2x_2x_4-x_3^2+x_4^2+x_4x_5=0 \\
x_1^2-x_1x_2+2x_1x_4+x_1x_5+2x_1x_6+x_1x_7-x_2x_7-x_3^2+x_4^2+x_4x_6=0 \\
x_1x_3-x_1x_4-x_1x_5-2x_1x_6-x_1x_7-2x_2x_3-2x_2x_4+2x_2x_7+x_3^2-x_4^2+ x_4x_7=0 \\
x_1^2+x_1x_3+x_1x_4-2x_1x_5+x_1x_7-x_2x_3-3x_2x_4-x_2x_5-x_2x_7-x_4^2+x_5^2=0 \\
-x_1^2+2x_1x_2-x_1x_4-x_1x_5-2x_1x_6-x_1x_7-x_2x_3-x_2x_4+x_2x_7+x_3^2-x_4^2+ x_5x_6=0 \\
-x_1x_3+x_1x_5+x_1x_6+2x_2x_3+2x_2x_4-x_2x_7-x_3^2+x_4^2+x_5x_7=0 \\
-x_1x_4-x_1x_5-x_1x_7+x_2x_3-x_2x_4-x_2x_6+2x_2x_7-x_3x_7-x_4^2+x_6^2=0.
\end{sistema}
\]
\[
\begin{array}{|cc|cc|}
\toprule
\text{Rational point} & \text{Disc.} & \text{Rational point} & \text{Disc.} \\
\midrule
(0:0:0:0:0:0:1) & \text{cusp} & (1:0:-1:0:0:0:0) & -12 \\
(1:1:1:-1:2:0:-1) & -67 & (0:1:0:0:0:1:0) & -8 \\
(0:0:0:1:-1:-1:1) & -43 & (0:1:0:0:0:0:0) & -7 \\
(0:1:2:-2:0:0:0) & -28 & (1:0:0:1:0:-1:-1) & -4 \\
(0:1:0:0:1:0:0) & -27 & (3:2:3:0:2:0:0) & -3 \\
(1:0:0:-1:0:-1:1) & -16 & & \\
\bottomrule
\end{array}
\]
\item Curve $X_0^+(197)$. Genus $g=6$. Equations of the canonical model in $\PP^{5}$
\[
\begin{sistema}
-x_1x_2+x_1x_4-x_1x_5-x_1x_6+x_2x_3=0 \\
x_1^2-x_1x_3-x_1x_4+x_1x_5+x_2x_4=0 \\
2x_1x_2-2x_1x_3-2x_1x_4+x_1x_5+x_1x_6-x_2^2-x_2x_6+x_3^2+x_3x_4+x_4^2=0 \\
-x_1^2+x_1x_2+x_1x_3+x_1x_6-x_2^2-x_2x_6+x_4x_5=0 \\
x_1x_2-x_1x_3-x_1x_4+x_3x_5+x_4x_6=0 \\
-x_1x_2+x_1x_3+x_1x_4-x_1x_6+x_2x_5-x_2x_6-x_3x_5+x_3x_6+x_5^2=0.
\end{sistema}
\]
\[
\begin{array}{|cc|cc|}
\toprule
\text{Rational point} & \text{Disc.} & \text{Rational point} & \text{Disc.} \\
\midrule
(0:0:0:0:0:1) & \text{cusp} & (1:1:1:0:0:0) & -19 \\
(1:-1:3:2:6:-6) & -163 & (1:0:0:0:-1:1) & -16 \\
(1:1:1:1:0:1) & -43 & (1:0:0:1:0:1) & -7 \\
(1:0:2:-1:0:-1) & -28 & (1:0:2:0:1:-1) & -4 \\
\bottomrule
\end{array}
\]
\item Curve $X_0^+(211)$. Genus $g=6$. Equations of the canonical model in $\PP^{5}$
\[
\begin{sistema}
-x_1^2-x_1x_4-x_1x_5+x_2x_3-x_2x_6+x_3x_4=0 \\
-x_1x_3-x_1x_4-x_1x_5+x_2x_3-x_2x_6+x_3x_5=0 \\
x_1^2-x_1x_2+x_1x_3+x_1x_4-x_2x_3+x_3x_6=0 \\
2x_1^2-2x_1x_2+x_1x_3+3x_1x_4+x_1x_5-2x_2x_3-x_2x_4+2x_2x_6+x_4^2=0 \\
x_1x_3+x_1x_4+2x_1x_5-x_1x_6-2x_2x_3-x_2x_5+2x_2x_6+x_4x_5=0 \\
x_1^2-2x_1x_2+x_1x_3+2x_1x_4+x_1x_5-x_1x_6-2x_2x_3-x_2x_4+x_2x_5+2x_2x_6-x_4x_6+ x_5^2=0.
\end{sistema}
\]
\[
\begin{array}{|cc|cc|}
\toprule
\text{Rational point} & \text{Disc.} & \text{Rational point} & \text{Disc.} \\
\midrule
(0:0:0:0:0:1) & \text{cusp} & (0:1:1:0:0:1) & -12 \\
(1:0:-1:-1:1:-1) & -67 & (0:0:1:0:0:0) & -8 \\
(2:1:2:0:0:-2) & -28 & (0:1:0:0:0:0) & -7 \\
(1:0:1:-1:-1:-1) & -27 & (2:-3:-1:-2:4:1) & -3 \\
\bottomrule
\end{array}
\]
\item Curve $X_0^+(223)$. Genus $g=6$. Equations of the canonical model in $\PP^{5}$
\[
\begin{sistema}
-x_1^2+x_1x_4-x_2x_3-x_2x_4+x_2x_5=0 \\
x_1^2+2x_1x_3-x_1x_4+x_1x_5+x_2^2+2x_2x_4+x_2x_6+x_3^2=0 \\
-x_1^2-x_1x_2-x_1x_3+x_1x_4-x_2x_4+x_2x_6+x_3x_4=0 \\
x_1^2-x_1x_2+2x_1x_5-x_1x_6+x_2^2+x_2x_4+2x_2x_6+x_3x_5=0 \\
-x_1^2-x_1x_2+x_1x_3+x_1x_4+x_2x_6-x_3x_6+x_4x_5=0 \\
-2x_1^2-2x_1x_2+x_1x_3+3x_1x_4+2x_1x_5+x_1x_6+x_2^2+4x_2x_6-x_3x_6-x_4^2-x_4x_6+ x_5^2=0.
\end{sistema}
\]
\[
\begin{array}{|cc|cc|}
\toprule
\text{Rational point} & \text{Disc.} & \text{Rational point} & \text{Disc.} \\
\midrule
(0:0:0:0:0:1) & \text{cusp} & (0:1:1:-1:0:0) & -12 \\
(2:-2:2:3:6:7) & -163 & (0:0:0:1:0:-1) & -11 \\
(1:0:-2:1:0:1) & -67 & (2:-3:-3:-1:-6:2) & -3 \\
(1:0:0:1:0:1) & -27 & & \\
\bottomrule
\end{array}
\]
\item Curve $X_0^+(229)$. Genus $g=7$. Equations of the canonical model in $\PP^{6}$
\[
\begin{sistema}
-x_1x_5+x_2x_4-x_2x_6+x_2x_7=0 \\
x_1x_2-x_1x_3-x_1x_4-x_1x_5+x_2x_3+x_2x_4-x_2x_6+x_3^2+x_3x_4=0 \\
-x_1x_5+x_2^2+x_2x_3+x_2x_4-x_2x_5-x_2x_6+x_3x_5=0 \\
x_1x_2-x_1x_3-x_1x_4-x_1x_6-x_2^2+x_2x_4+x_3^2+x_3x_6=0 \\
x_1x_2+x_1x_3-x_1x_6-x_1x_7-x_2^2-x_2x_3+x_2x_6+x_3x_7=0 \\
-2x_1x_2+x_1x_4+x_1x_5+x_1x_6+x_1x_7-x_2x_3-x_2x_4+x_2x_6-x_3^2+x_4^2=0 \\
-x_1x_3+x_1x_6-x_2^2-x_2x_3-x_2x_4+x_2x_5+x_4x_5=0 \\
-x_1x_2+x_1x_5+x_1x_7+x_2x_4-x_2x_5+x_5x_6=0 \\
x_1x_2-2x_1x_5-x_1x_6-x_1x_7+x_2x_3+2x_2x_4-x_2x_5-x_2x_6+x_3^2-x_4x_6+x_5x_7=0 \\
x_1x_5-x_2x_3-x_2x_4-x_3^2-x_4x_7+x_6^2=0.
\end{sistema}
\]
\[
\begin{array}{|cc|cc|}
\toprule
\text{Rational point} & \text{Disc.} & \text{Rational point} & \text{Disc.} \\
\midrule
(0:0:0:0:0:0:1) & \text{cusp} & (0:1:-1:0:0:0:0) & -12 \\
(1:1:0:1:1:0:0) & -43 & (1:0:0:0:0:0:0) & -11 \\
(1:0:1:0:0:1:0) & -27 & (2:0:1:-1:0:1:0) & -4 \\
(0:0:0:0:1:0:0) & -19 & (2:-3:-1:0:-6:-4:0) & -3 \\
(0:0:1:-1:0:-1:0) & -16 & & \\
\bottomrule
\end{array}
\]
\item Curve $X_0^+(233)$. Genus $g=7$. Equations of the canonical model in $\PP^{6}$
\[
\begin{sistema}
-x_1x_2-x_1x_3-x_1x_4+x_2^2-x_2x_3-x_2x_5-x_2x_6+x_2x_7=0 \\
x_1x_2+x_1x_4-x_2^2+x_2x_6+x_3x_5=0 \\
x_1^2-2x_1x_2-x_1x_3+x_1x_5+2x_2^2-2x_2x_3-x_2x_4-2x_2x_5-x_2x_6+x_3x_6=0 \\
-2x_1x_3-x_1x_6+x_2^2-x_2x_3-x_2x_4-x_2x_5+x_3x_7=0 \\
x_1^2-2x_1x_2+x_1x_4+x_1x_5+2x_2^2-2x_2x_3-2x_2x_5-x_2x_6+x_3x_4+x_4^2=0 \\
x_1x_3+x_1x_6-x_2^2+x_2x_3+x_2x_5+x_4x_5=0 \\
x_1x_2+x_1x_3-x_1x_7-x_2^2+x_2x_3-x_2x_4+x_2x_5+x_2x_6+x_4x_6=0 \\
-x_1x_2-x_1x_3-x_1x_4+x_1x_5-x_1x_6-x_1x_7+2x_2^2-x_2x_3-x_2x_4-2x_2x_5-x_2x_6+ x_5^2=0 \\
-x_1^2+x_1x_2+2x_1x_3-x_1x_5+x_1x_6+x_1x_7-2x_2^2+2x_2x_3+x_2x_4+2x_2x_5+x_4x_7+ x_5x_6=0 \\
-x_1^2+x_1x_2-x_1x_5+x_2x_4-x_2x_6+x_5x_7+x_6^2=0.
\end{sistema}
\]
\[
\begin{array}{|cc|cc|}
\toprule
\text{Rational point} & \text{Disc.} & \text{Rational point} & \text{Disc.} \\
\midrule
(0:0:0:0:0:0:1) & \text{cusp} & (0:0:1:0:0:0:0) & -8 \\
(1:0:2:-2:1:0:2) & -28 & (1:0:0:0:-1:0:0) & -7 \\
(0:0:1:-1:0:0:0) & -19  & (1:1:0:-2:1:2:1) & -4 \\
(1:1:0:0:1:0:1) & -16 & & \\
\bottomrule
\end{array}
\]
\item Curve $X_0^+(241)$. Genus $g=7$. Equations of the canonical model in $\PP^{6}$
\[
\begin{sistema}
x_1^2+x_1x_4-x_2^2-x_2x_4+x_2x_6=0 \\
x_1x_5-x_2x_3-x_2x_5+x_2x_7=0 \\
x_1x_2+2x_1x_3+x_1x_6-x_1x_7-x_2^2+x_3^2=0 \\
x_1x_3-x_1x_7+x_2x_4+x_3x_4=0 \\
-x_1x_4-x_1x_5-x_2x_4+x_2x_5-x_3x_5-x_4^2+x_4x_5=0 \\
-x_1^2+x_1x_2-x_1x_3+x_1x_4+x_1x_5+x_1x_7+x_2x_3-2x_2x_4-x_2x_5+x_3x_5-x_3x_6+ x_4x_6=0 \\
-x_1^2-x_1x_3+x_1x_4-x_1x_6+x_2^2+x_2x_3-x_2x_4+x_3x_5-x_3x_6-x_3x_7+x_4x_7=0 \\
-x_1x_2+x_1x_3-3x_1x_4-x_1x_5+x_1x_6+x_2^2+x_2x_4-x_3x_5+x_3x_6-x_4^2+x_5^2=0 \\
-x_1^2-x_1x_3+x_1x_4+x_1x_5-x_1x_6+x_2^2+x_2x_3-2x_2x_4-x_2x_5+x_3x_5-x_3x_6- x_3x_7+x_5x_6=0 \\
-x_1^2+x_1x_2-x_1x_3+x_1x_4+2x_1x_5+x_1x_7+x_2x_3-2x_2x_4-2x_2x_5+x_5x_7+ x_6^2=0.
\end{sistema}
\]
\[
\begin{array}{|cc|cc|}
\toprule
\text{Rational point} & \text{Disc.} & \text{Rational point} & \text{Disc.} \\
\midrule
(0:0:0:0:0:0:1) & \text{cusp} & (1:1:0:0:1:0:0) & -12 \\
(1:0:-2:-1:0:0:0) & -67 & (0:0:0:1:1:0:0) & -8 \\
(0:1:-1:0:0:1:-1) & -27 & (1:-1:-2:0:0:0:-2) & -4 \\
(1:1:0:0:0:0:0) & -16 & (3:1:-4:-6:-3:4:2) & -3 \\
\bottomrule
\end{array}
\]
\item Curve $X_0^+(257)$. Genus $g=7$. Equations of the canonical model in $\PP^{6}$
\[
\begin{sistema}
-x_1^2-x_1x_2+x_1x_4-x_2^2-x_2x_4+x_2x_5=0 \\
-x_1^2-2x_1x_2+x_1x_3+x_1x_4-x_1x_5-2x_2^2-x_2x_4+x_2x_6=0 \\
-x_1x_2-x_1x_6-x_2x_3+x_2x_7=0 \\
x_1x_2-x_1x_3+x_1x_6+x_2^2+x_2x_4+x_3x_4=0 \\
-x_1x_3-x_1x_4+x_1x_6-x_1x_7+x_2^2-x_2x_3+x_2x_4+x_3x_5=0 \\
x_1^2+x_1x_2-x_1x_3-x_1x_4+2x_2^2-x_2x_3+2x_2x_4-x_3x_6+x_4x_5=0 \\
x_1^2+x_1x_2-2x_1x_4+x_1x_5-x_1x_7+2x_2^2-x_2x_3+x_2x_4+x_3^2-x_3x_6-x_3x_7+ x_4x_6=0 \\
x_1^2-x_1x_3-x_1x_4-x_1x_5-x_1x_6+x_2^2-x_2x_3+x_2x_4-x_3^2-x_3x_6+x_3x_7+ x_5^2=0 \\
x_1^2+x_1x_3-2x_1x_4-2x_1x_7+x_2^2-x_2x_3+x_2x_4+x_3^2-x_3x_6-x_3x_7+x_4x_7+ x_5x_6=0 \\
-2x_1x_2+2x_1x_3-2x_1x_6-x_1x_7-2x_2^2-x_2x_4-x_3x_6+x_4x_7-x_5x_7+x_6^2=0.
\end{sistema}
\]
\[
\begin{array}{|cc|cc|}
\toprule
\text{Rational point} & \text{Disc.} & \text{Rational point} & \text{Disc.} \\
\midrule
(0:0:0:0:0:0:1) & \text{cusp} & (0:0:1:0:0:0:1) & -11 \\
(1:0:0:1:0:0:-1) & -67 & (0:0:0:1:0:0:0) & -8 \\
(0:1:0:-1:0:1:0) & -16 & (2:-1:0:1:0:1:0) & -4 \\
\bottomrule
\end{array}
\]
\item Curve $X_0^+(269)$. Genus $g=6$. Equations of the canonical model in $\PP^{5}$
\[
\begin{sistema}
x_1x_4+x_2^2+x_2x_3+x_3^2+x_3x_5=0 \\
-x_1x_2-x_1x_5+x_2^2+x_2x_3+x_2x_5+x_3x_6=0 \\
x_1x_2+2x_1x_4+x_1x_6+x_3^2+x_3x_4+x_4^2=0 \\
x_1x_2-x_1x_4+x_1x_5-x_2^2-x_2x_3+x_2x_4+x_3x_4+x_4x_5=0 \\
x_1x_2+x_1x_3+x_1x_4+x_1x_5+x_1x_6-x_2x_3-x_2x_5+x_4x_6=0 \\
-x_1x_2-x_1x_3-2x_1x_4-x_1x_5-x_1x_6-x_2^2+2x_2x_5-x_2x_6-x_3^2+x_5^2=0.
\end{sistema}
\]
\[
\begin{array}{|cc|cc|}
\toprule
\text{Rational point} & \text{Disc.} & \text{Rational point} & \text{Disc.} \\
\midrule
(0:0:0:0:0:1) & \text{cusp} & (1:1:0:-1:0:0) & -16 \\
(1:-1:0:-1:1:2) & -67 & (1:0:0:0:0:0) & -11 \\
(1:1:0:-1:-1:0) & -43 & (1:-1:2:-3:0:0) & -4 \\
\bottomrule
\end{array}
\]
\item Curve $X_0^+(271)$. Genus $g=6$. Equations of the canonical model in $\PP^{5}$
\[
\begin{sistema}
x_1^2+2x_1x_2+x_1x_5-x_1x_6-x_2x_4-2x_2x_5-x_2x_6+x_3x_5=0 \\
x_1^2+x_1x_2-x_1x_4-x_1x_6-x_2x_3-x_2x_4-x_2x_5-x_2x_6+x_4^2=0 \\
-x_1^2-2x_1x_2-x_1x_5+x_1x_6+x_2x_3+x_2x_4+x_2x_5+x_4x_5=0 \\
x_1x_2-x_1x_4-x_1x_6-x_2x_3+x_2x_6+x_3^2+x_4x_6=0 \\
x_1^2+3x_1x_2+2x_1x_3+x_1x_5-3x_1x_6-2x_2x_4-x_2x_5+x_5^2=0 \\
-x_1x_2-x_1x_3-x_1x_4-x_1x_5+x_1x_6-x_2x_3+x_2x_4+x_2x_5+x_2x_6+x_3x_4+x_5x_6=0.
\end{sistema}
\]
\[
\begin{array}{|cc|cc|}
\toprule
\text{Rational point} & \text{Disc.} & \text{Rational point} & \text{Disc.} \\
\midrule
(0:0:0:0:0:1) & \text{cusp} & (0:1:0:0:0:0) & -19 \\
(1:0:1:1:0:1) & -43 & (1:0:1:0:0:1) & -12 \\
(0:1:1:1:-1:0) & -27 & (3:-2:-5:4:-4:-3) & -4 \\
\bottomrule
\end{array}
\]
\item Curve $X_0^+(281)$. Genus $g=7$. Equations of the canonical model in $\PP^{6}$
\[
\begin{sistema}
-x_1x_2+x_1x_3+x_2x_3+x_2x_4-x_3x_4+x_3x_6=0 \\
-x_1x_6-x_2x_3-x_2x_4-x_2x_5+x_3^2+x_3x_5+x_3x_7=0 \\
x_1x_2+x_1x_7+2x_3x_4-x_3x_5+x_4^2=0 \\
-x_1x_2-2x_1x_3-x_1x_4-x_1x_5-x_1x_7-x_2x_3+2x_3x_5+x_4x_5=0 \\
x_1x_2-x_1x_3+x_2x_5+2x_3x_4+x_4x_6=0 \\
2x_1x_3+2x_1x_4+x_1x_5+x_1x_6+x_1x_7-x_2^2+x_2x_3+x_2x_4+x_2x_5+x_3x_4-2x_3x_5+ x_4x_7=0 \\
x_1x_2+2x_1x_3+x_1x_4+x_1x_5+x_1x_6+x_1x_7+x_2x_3-x_2x_4-x_3x_5+x_5^2=0 \\
x_1x_2-x_1x_3+x_2^2-x_2x_3-x_2x_4-x_2x_5+x_3x_5+x_5x_6=0 \\
x_1x_2-x_1x_3-x_1x_4-x_1x_7+x_2^2+x_2x_6+x_3x_5+x_5x_7=0 \\
x_1x_2-x_1x_3+x_1x_6-x_2x_3-x_2x_4-x_2x_7+2x_3x_4+x_6^2=0.
\end{sistema}
\]
\[
\begin{array}{|cc|cc|}
\toprule
\text{Rational point} & \text{Disc.} & \text{Rational point} & \text{Disc.} \\
\midrule
(0:0:0:0:0:0:1) & \text{cusp} & (1:0:0:0:0:0:0) & -16 \\
(2:-5:-5:-1:1:5:-3) & -163 & (1:0:0:1:0:0:-1) & -8 \\
(0:1:-1:1:1:1:-1) & -43 & (0:0:1:0:0:0:-1) & -7 \\
(0:0:1:-2:0:-2:-1) & -28 & (1:2:-2:2:2:0:-2) & -4 \\
\bottomrule
\end{array}
\]
\item Curve $X_0^+(359)$. Genus $g=6$. Equations of the canonical model in $\PP^{5}$
\[
\begin{sistema}
x_1x_2 -x_1x_5 + x_1x_6 + x_3x_4=0 \\
x_1^2 + x_1x_2 + x_1x_3 -2x_1x_4 - x_1x_5 + x_2x_4 + x_3x_5=0 \\
-3x_1x_2 -x_1x_3 + 2x_1x_4 + 2x_1x_5 - x_1x_6 + x_2^2 -x_2x_3 - x_2x_4-x_3^2 + x_3x_6=0 \\
x_1^2 + 3x_1x_2 + x_1x_3 -2x_1x_4 -2x_1x_5 + x_1x_6-x_2^2 +x_2x_4 + x_3^2 + x_4x_5=0 \\
-x_1^2 -2x_1x_2 -2x_1x_3 + 3x_1x_4 + 3x_1x_5 - x_6 + x_2^2 -2x_2x_4 - x_2x_5 -x_3^2 + x_4x_6=0 \\
x_1^2 +2x_1x_3 -2x_1x_4 - 2x_1x_5 + x_2^2 +x_2x_4 - x_2x_6 + x_5^2=0.
\end{sistema}
\]
\[
\begin{array}{|cc|cc|}
\toprule
\text{Rational point} & \text{Disc.} & \text{Rational point} & \text{Disc.} \\
\midrule
(0:0:0:0:0:0:1) & \text{cusp} & (1:0:0:0:0:0:0) & -28 \\
(2:-5:-5:-1:1:5:-3) & -163 & (1:0:0:1:0:0:-1) & -19 \\
(0:1:-1:1:1:1:-1) & -67 & (0:0:1:0:0:0:-1) & -7 \\
(0:0:1:-2:0:-2:-1) & -43 & & \\
\bottomrule
\end{array}
\]
\end{itemize}

\subsection{Modular parametrization}\label{subsec:modpar}

Here we list the isogeny factor $E$ of $J_0^+(p)$, a generator $P_0$ of the Mordell Weil group, an explicit formula for the modular parametrization~$\phi_+\colon X_0^+(p)\to E$ and the modular degree $\deg(\phi_+)$ of the morphism $\phi_+$.
\begin{itemize}
\item Curve $X_0^+(163)$. Elliptic curve $E: y^2+y=x^3-2x+1$. \\
\noindent
Mordell-Weil generator $P_0=(1,0)$. Modular degree $\deg(\phi_+)=3$. \\
\noindent
Modular parametrization: $\phi_+(x_1:x_2:x_3:x_4:x_5:x_6)=(P_x,P_y)$,
\[
P_x =\frac{-2x_3-x_4+x_5-x_6}{-x_3+x_5}; \, P_y =\frac{-x_1^2+x_1x_2+2x_1x_3+ 2x_1x_4+2x_1x_6-x_3^2+x_4^2}{x_1x_3-x_1x_5}.
\]

\item Curve $X_0^+(197)$. Elliptic curve $E: y^2+y=x^3-5x+4$. \\
\noindent
Mordell-Weil generator $P_0=(1,0)$. Modular degree $\deg(\phi_+)=5$. \\
\noindent
Modular parametrization: $\phi_+(x_1:x_2:x_3:x_4:x_5:x_6)=(P_x,P_y)$,
\begin{align*}
\num(P_x)&=-38x_1^2-25x_1x_2-6x_1x_3+13x_1x_4-61x_1x_5+11x_1x_6+22x_2^2+ \\
&+10x_2x_5+20x_2x_6+17x_3x_5-11x_3x_6, \\
\den(P_x)&=-36x_1^2-34x_1x_2+11x_1x_3+30x_1x_4-27x_1x_5-8x_1x_6+12x_2^2+ \\
&+5x_2x_5+11x_2x_6, \\
\num(P_y) &=1715x_1^3-899x_1^2x_2-2691x_1^2x_3-1013x_1^2x_4-2889x_1^2x_5-81x_1^2x_6+ \\
&+1875x_1x_2^2+2476x_1x_2x_5+1176x_1x_2x_6-669x_1x_3x_5+1224x_1x_3x_6+ \\
&+2205x_1x_5x_6-1306x_1x_6^2-384x_2^2x_5-384x_2x_5x_6, \\
\den(P_y)&= 961x_1^3-49x_1^2x_2+3875x_1^2x3-632x_1^2x_4+1195x_1^2x_5+356x_1^2x_6+ \\
&+593x_1x_2^2+655x_1x_2x_5+432x_1x_2x_6-839x_1x_3^2-1680x_1x_3x_4+ \\
&+154x_1x_3x_5+384x_1x_5^2.
\end{align*}

\item Curve $X_0^+(229)$. Elliptic curve $E: y^2+xy=x^3-2x-1$. \\
\noindent
Mordell-Weil generator $P_0=(-1,1)$. Modular degree $\deg(\phi_+)=4$. \\
\noindent
Modular parametrization: $\phi_+(x_1:x_2:x_3:x_4:x_5:x_6:x_7)=(P_x,P_y)$,
\begin{align*}
P_x &=\frac{-x_2-x_5+x_6}{x_2+x_5}, \\
\num(P_y)&=x_1^2+3x_1x_2+7x_1x_3+6x_1x_4+x_1x_5-x_1x_6-4x_1x_7+ \\
&-2x_2x_3-2x_2x_4+3x_2x_6-2x_3^2+x_4x_6-2x_4x_7-x_6x_7, \\
\den(P_y)&=x_1^2+5x_1x_2+x_1x_4+2x_1x_5-x_1x_6- x_1x_7-x_3^2-x_4x_7+x_6^2.
\end{align*}

\item Curve $X_0^+(269)$. Elliptic curve $E: y^2+y=x^3-2x-1$. \\
\noindent
Mordell-Weil generator $P_0=(-1,0)$. Modular degree $\deg(\phi_+)=3$. \\
\noindent
Modular parametrization: $\phi_+(x_1:x_2:x_3:x_4:x_5:x_6)=(P_x,P_y)$,
\[
P_x =\frac{x_4+x_6}{x_2}, \qquad P_y =\frac{x_1x_5-x_2x_3+x_3x_4}{x_1x_2}.
\]

\item Curve $X_0^+(359)$. Elliptic curve $E: y^2+xy+y=x^3-x^2-7x+8$. \\
\noindent
Mordell-Weil generator $P_0=(2,-1)$. Modular degree $\deg(\phi_+)=4$. \\
\noindent
Modular parametrization: $\phi_+(x_1:x_2:x_3:x_4:x_5:x_6)=(P_x,P_y)$,
\begin{align*}
P_x &=\frac{-4x_1^2 + 3x_1x_2 -4x_1x_3 -x_1x_5 + 3x_1x_6 + x_2x_4 + x_4^2}{-2x_1^2+ x_1x_2 -2x_1x_3 + x_1x_6}, \\
P_y &=\frac{-5x_1^2 + 4x_1x_2 -6x_1x_3 + 5x_1x_4 + 4x_1x_6 -x_2x_4 -x_2x_5 - x_2x_6 + x_3^2 + 2x_4^2}{3x_1^2 + 3x_1x_3 -2x_1x_4 - 2x_1x_5 + x_2x_4}.
\end{align*}
\end{itemize}

\end{document}